\documentclass[a4paper, reqno]{amsart}
\usepackage{amssymb, amsmath, amsthm, chngcntr, color, enumerate, tikz}

\usepackage{mathabx} 

\newtheorem{theorem}{Theorem}
\newtheorem{corollary}[theorem]{Corollary}
\newtheorem{lemma}[theorem]{Lemma}
\newtheorem{proposition}[theorem]{Proposition}

\newtheorem{conjecture}[theorem]{Conjecture}

\theoremstyle{definition}

\newtheorem{remark}[theorem]{Remark}

\newtheorem*{acknowledgement}{Acknowledgement}
\newtheorem*{notation}{Notation}

\newcommand{\C}{\mathbb{C}}
\newcommand{\R}{\mathbb{R}}
\newcommand{\Z}{\mathbb{Z}}
\newcommand{\N}{\mathbb{N}}
\newcommand{\T}{\mathbb{T}}

\newcommand{\ud}{\mathrm{d}}

\title{Uniform $L^p$ Resolvent Estimates on the Torus}
\date{}


\author{Jonathan Hickman}
\address{Mathematical Institute, University of St Andrews, North Haugh, St Andrews, Fife, KY16 9SS, UK.}
\email{jeh25@st-andrews.ac.uk}

\begin{document}

\begin{abstract} 
 A new range of uniform $L^p$ resolvent estimates is obtained in the setting of the flat torus, improving previous results of Bourgain, Shao, Sogge and Yao. The arguments rely on the $\ell^2$-decoupling theorem and multidimensional Weyl sum estimates. 
\end{abstract}

\maketitle




\section{Introduction}

This article continues a line of investigation pursued by Dos Santos Ferriera, Kenig and Salo~\cite{DKS} and Bourgain, Shao, Sogge and Yao~\cite{BSSY} concerning uniform $L^p$ estimates for resolvents of Laplace--Beltrami operators on compact manifolds. Here new bounds are obtained only in the special case of the flat $n$-dimensional torus $\T^n := \R^n \setminus \Z^n$ but, in order to contextualise the results, it is useful to recall the general setup from~\cite{DKS, BSSY}. To this end, let $(M,g)$ be a smooth, compact manifold of dimension $n \geq 3$ without boundary and $\Delta_g$ be the associated Laplace--Beltrami operator. In~\cite{DKS} the following problem was introduced: determine the regions $\mathcal{R} \subseteq \C$ for which there is a uniform bound
\begin{equation}\label{uniform resolvent}
    \|u\|_{L^{\frac{2n}{n-2}}(M)} \leq C_{\mathcal{R}} \|(\Delta_g + z)u\|_{L^{\frac{2n}{n+2}}(M)} \qquad \textrm{for all $z \in \mathcal{R}$.}
\end{equation}

Interest in inequalities of the form~\eqref{uniform resolvent} was partly inspired by earlier work on the standard Laplacian on $n$-dimensional euclidean space. In the euclidean setting scaling considerations imply that $(\frac{2n}{n+2}, \frac{2n}{n-2})$ is the only exponent pair lying on the line of duality for which~\eqref{uniform resolvent} is meaningful; this observation also motivates the choice of Lebesgue exponents featured above. Moreover, it was shown by Kenig, Ruiz and Sogge~\cite{KRS} that the euclidean analogue of~\eqref{uniform resolvent} holds for $\mathcal{R} = \C$. By contrast, such uniformity in~\eqref{uniform resolvent} patently fails for compact manifolds $(M,g)$: in this case $-\Delta_g$ has a discrete spectrum and therefore~\eqref{uniform resolvent} cannot hold whenever $z$ is an eigenvalue of $-\Delta_g$. It is therefore natural when working in the compact manifold setting to consider regions $\mathcal{R}$ which are bounded away from the non-negative real line, and thereby avoid the spectrum. 

As in \cite{BSSY}, it is convenient to write $z = (\lambda + i\mu)^2$ for some $\lambda, \mu \in \R$ and express the results in terms of these real parameters. For $\lambda \leq 1$ the situation is relatively easy to understand and is treated in \cite[\S 2]{BSSY}. Henceforth, it is assumed that $\lambda \geq 1$. The problem is to determine how small $|\mu|$ can be (in terms of $\lambda$) whilst retaining uniformity in \eqref{uniform resolvent}. 

\begin{theorem}\label{resolvent theorem} Let $n \geq 3$ and $\Delta_{\T^n}$ be the Laplacian on the flat torus $\T^n := \R^n /\Z^n$. For all $\varepsilon > 0$ the uniform $L^p$ resolvent bound 
\begin{equation*}
    \|u\|_{L^{\frac{2n}{n-2}}(\T^n)} \leq C_{\varepsilon} \| (\Delta_{\T^n} + z)u\|_{L^{\frac{2n}{n+2}}(\T^n)}
\end{equation*}
holds whenever $z \in \C$ belongs to the region
\begin{equation*}
    \mathcal{R}_{\mathrm{new}} := \big\{ z = (\lambda + i \mu)^2  \in \C : \lambda, \mu \in \R, \lambda \geq 1, \, |\mu| \geq \lambda^{-1/3 + \varepsilon} \big\}.
\end{equation*}
\end{theorem}

It is useful to compare the theorem with existent results. Shen~\cite{Shen2001} previously showed that Theorem~\ref{resolvent theorem} holds in the more restrictive region
\begin{equation*}
\mathcal{R}_{\mathrm{DKSS}} := \big\{ z = (\lambda + i \mu)^2  \in \C : \lambda, \mu \in \R, \lambda \geq 1, \, |\mu| \geq 1 \big\}.
\end{equation*}
and this was later generalised to arbitrary compact manifolds by Dos Santos Ferriera, Kenig and Salo~\cite{DKS}. In~\cite{DKS} it was also asked whether it is possible to extend the uniform bounds beyond $\mathcal{R}_{\mathrm{DKSS}}$ for general manifolds. Interestingly, Bourgain, Shao, Sogge and Yao~\cite{BSSY} showed that the region $\mathcal{R}_{\mathrm{DKSS}}$ is, in fact, optimal in the case of Zoll manifolds (one example being the standard euclidean sphere $S^{n}$), in the sense that here it is not possible to relax $|\mu| \geq 1$ to $|\mu| \geq \lambda^{-\alpha}$ for any $\alpha > 0$ in $\mathcal{R}_{\mathrm{DKSS}}$. Underpinning such behaviour in the Zoll case is the tight spectral clustering exhibited by $-\Delta_g$. Clustering does not occur for the torus and, consequently, improvements may be obtained for $\T^n$. Indeed, in~\cite{BSSY} it was shown that for all $\varepsilon > 0$ Theorem~\ref{resolvent theorem} holds for the region
\begin{equation*}
\mathcal{R}_{\mathrm{BSSY}} := \big\{ z = (\lambda + i \mu)^2  \in \C : \lambda, \mu \in \R, \lambda \geq 1, \, |\mu| \geq \lambda^{-\varepsilon_n + \varepsilon} \big\}.
\end{equation*}
where $\varepsilon_n >0$ is given by
\begin{equation*}
    \varepsilon_n := \frac{2(n-1)}{n(n+1)} \textrm{ if $n \geq 3$ is odd}, \quad \varepsilon_n := \frac{2(n-1)}{n^2 +2n+2} \textrm{ if $n \geq 4$ is even};
\end{equation*}
furthermore, by using additional number theoretic input, it was also shown in~\cite{BSSY} that for $n = 3$ the slightly relaxed condition $\varepsilon_3 := \frac{85}{252}$ is sufficient. 

\begin{figure}
    \centering

\begin{tikzpicture}[scale=1]

{
\draw[red] (1,2)
\foreach \t in {1.1,1.2,...,6}
    {
        --({\t},{2*\t^(-0.4)})
    }
    node[anchor=west]{$\gamma_{\mathrm{new}}$};
    }

{
 \draw[red] (1,-2)
 \foreach \t in {1.1,1.2,...,6}
    {
        --({\t},{-2*\t^(-0.4)})
    }
    ;
    
    }

{
\draw[blue] (1,2)
\foreach \t in {1.1,1.2,...,6}
    {
        --({\t},{2})
    }
   node[anchor=west]{$\gamma_{\mathrm{DKSS}}$}; ;
    }
{
\draw[blue] (1,-2)
\foreach \t in {1.1,1.2,...,6}
    {
        --({\t},{-2})
    }
   ;
    }

    {
\draw[orange] (1,2)
\foreach \t in {1.1,1.2,...,6}
    {
        --({\t},{2*\t^(-0.15)})
    }
    node[anchor=west]{$\gamma_{\mathrm{BSSY}}$};
    }  
{    
\draw[orange] (1,-2)
\foreach \t in {1.1,1.2,...,6}
    {
        --({\t},{-2*\t^(-0.15)})
    }
    ;
    }

      {
\draw[green] (1,2)
\foreach \t in {1.1,1.2,...,6}
    {
        --({\t},{2*\t^(-1)})
    }
    node[anchor=west]{$\gamma_{\mathrm{opt}}$};
    }  
 
{    
\draw[green] (1,-2)
\foreach \t in {1.1,1.2,...,6}
    {
        --({\t},{-2*\t^(-1)})
    }
    ;
    }

  \draw[->][black, thick] 
		(-0.5,0)->(7,0) node[anchor=west]{$\lambda$};
\draw[->][black, thick] 
        (0,-2.5)->(0,2.5) node[anchor=east]{$\mu$}; 
\draw[dotted] 
        (1,-2.5)->(1,2.5)
        ; 

\draw[-][black, thick] 
        (1,-0.1)->(1,0.1)  node[pos=-1]{\Large 1} ; 
\draw[-][black, thick] 
        (-0.1,2)->(0.1,2) node[pos=-1]{\Large 1}; 
 \draw[-][black, thick] 
        (-0.1,-2)->(0.1,-2) node[pos=-1]{\Large -1};   
\end{tikzpicture}

    \caption{Successive results and the optimal region. Each curve ${\color{blue}\gamma_{\mathrm{DKSS}}}$, ${\color{orange}\gamma_{\mathrm{BSSY}}}$, ${\color{red}\gamma_{\mathrm{new}}}$ and ${\color{green}\gamma_{\mathrm{opt}}}$ corresponds to the interesting part of the boundary of  ${\color{blue}\mathcal{R}_{\mathrm{DKSS}}}$, ${\color{orange}\mathcal{R}_{\mathrm{BSSY}}}$, ${\color{red}\mathcal{R}_{\mathrm{new}}}$ and ${\color{green}\mathcal{R}_{\mathrm{opt}}}$, respectively, in the coordinates $(\lambda, \mu)$.}
    \label{figure}
\end{figure}

Theorem~\ref{resolvent theorem} provides a further improvement over the ranges $\mathcal{R}_{\mathrm{DKSS}}$ and $\mathcal{R}_{\mathrm{BSSY}}$ (at least for $n > 3$); see Figure~\ref{figure}. Note for $n=3$ the numerology of the new result coincides with the $\frac{2(n-1)}{n(n+1)}$ exponent from~\cite{BSSY}. A pleasant feature of Theorem~\ref{resolvent theorem} is that $\mathcal{R}_{\mathrm{new}}$ provides a ``uniform'' strengthening over $\mathcal{R}_{\mathrm{DKSS}}$ in all dimensions.  

It is remarked that $\mathcal{R}_{\mathrm{new}}$ is certainly not sharp and a natural conjecture would be the following.

\begin{conjecture}\label{resolvent conjecture} Let $n \geq 3$ and $\Delta_{\T^n}$ be the Laplacian on the flat torus $\T^n := \R^n /\Z^n$. For all $\varepsilon > 0$ the uniform $L^p$ resolvent bound 
\begin{equation*}
    \|u\|_{L^{\frac{2n}{n-2}}(\T^n)} \leq C_{\varepsilon} \| (\Delta_{\T^n} + z)u\|_{L^{\frac{2n}{n+2}}(\T^n)}
\end{equation*}
holds whenever $z \in \C$ belongs to the region
\begin{equation*}
    \mathcal{R}_{\mathrm{opt}} := \big\{ z = (\lambda + i \mu)^2  \in \C : \lambda, \mu \in \R, \lambda \geq 1, \, |\mu| \geq \lambda^{-1 + \varepsilon} \big\}.
\end{equation*}
\end{conjecture}

 A slightly larger region, given by taking $\varepsilon = 0$ in the definition of  $\mathcal{R}_{\mathrm{opt}}$, featured in the original question posed in~\cite{DKS}. Conjecture~\ref{resolvent conjecture} is closely related to the so-called \textit{discrete restriction conjecture} for the sphere studied in~\cite{Bourgain1993}, which partially motivates the above definition of $\mathcal{R}_{\mathrm{opt}}$; this connection is discussed in more detail in \S\ref{spectral projector section} below.

The proof of Theorem~\ref{resolvent theorem} follows the strategy of~\cite{BSSY} but takes advantage of new estimates available due to the Bourgain--Demeter $\ell^2$-decoupling theorem~\cite{BD2015}. In~\cite{BSSY} uniform resolvent estimates were shown to be equivalent to $L^{\frac{2n}{n+2}} \to L^{\frac{2n}{n-2}}$ bounds for certain spectral projectors with thin bandwidths; the precise details of this equivalence are recalled in \S\ref{spectral projector section}. The desired spectral projection bounds are then proved using the $\ell^2$-decoupling inequality. It is not surprising that decoupling should play a r\^ole here since it has already had numerous applications to the spectral theory of $\Delta_{\T^n}$~\cite{Bourgain2013, BD2015, BSSY}. 

The Bourgain--Demeter theorem yields an $L^{\frac{2(n+1)}{n+3}} \to L^{\frac{2(n+1)}{n-1}}$ bound for the projector; see Corollary~\ref{decoupling corollary 1} below. Roughly speaking, to obtain the desired $L^{\frac{2n}{n+2}} \to L^{\frac{2n}{n-2}}$ inequality, one interpolates the $L^{\frac{2(n+1)}{n+3}} \to L^{\frac{2(n+1)}{n-1}}$ estimate with an $L^1 \to L^{\infty}$ estimate. The $L^{\infty}$ bound for the projector follows from a pointwise estimate for the kernel which, as in~\cite{BSSY}, is established using the classical lattice point counting method of Hlawka~\cite{Hlawka1950} (see also~\cite[Chapter 1]{Sogge2017}).   

Hlawka's original argument~\cite{Hlawka1950} has been refined by numerous authors (see, for instance,~\cite{KN1991, KN1992, Muller1999, Guo2012}). In~\cite{BSSY} exponential sum bounds from~\cite{KN1992} were applied to yield the slightly improved exponent $\varepsilon_3 = \frac{85}{252}$ mentioned above. Similarly, by applying a more refined analysis involving the multidimensional Weyl sum estimates from~\cite{Muller1999}, it is possible to slightly extend $\mathcal{R}_{\mathrm{new}}$ in all dimensions.

\begin{theorem}\label{refined resolvent theorem} For $n \geq 3$ and all $\varepsilon > 0$ the result of Theorem~\ref{resolvent theorem} holds for 
\begin{equation*}
\mathcal{R}_{\mathrm{new}}' := \big\{ z = (\lambda + i \mu)^2  \in \C : \lambda, \mu \in \R, \lambda \geq 1, \, |\mu| \geq \lambda^{-\beta_n + \varepsilon} \big\}
\end{equation*}
where
\begin{equation*}
    \beta_n := \frac{1}{3} + \frac{n}{3} \cdot \frac{1}{21n^2 - n - 24}.
\end{equation*}
\end{theorem}

Taking $n = 3$ the exponent becomes $\beta_3 = \frac{55}{162}$ which is slightly larger than the previous best exponent $\varepsilon_3 = \frac{85}{252}$ from~\cite{BSSY}. This improvement for $n=3$ is due in part to the use of stronger multidimensional Weyl sum estimates from~\cite{Muller1999} (as opposed to the estimates of~\cite{KN1992} used in~\cite{BSSY}) and also due in part to the use of the $\ell^2$-decoupling inequality, which allows for greater leverage of the exponential sum bounds.  

This article is structured as follows:
\begin{itemize}
    \item In \S\ref{spectral projector section} preliminary results from~\cite{BSSY} and, in particular, the details of the equivalence between resolvent and spectral projection estimates, are reviewed.
    \item In \S\ref{proof section} spectral projection bounds are proven, following the scheme described above. Using the equivalence discussed in \S\ref{spectral projector section}, this provides the proof of Theorem~\ref{resolvent theorem}.
    \item In \S\ref{Weyl sum section} exponential sum estimates from~\cite{Muller1999} are applied to refine the argument from \S\ref{proof section}, yielding Theorem~\ref{refined resolvent theorem}.
\end{itemize}

\begin{notation}  Given positive numbers $A, B \geq 0$ and a list of objects $L$, the notation $A \lesssim_L B$, $B \gtrsim_L A$ or $A = O_L(B)$ signifies that $A \leq C_L B$ where $C_L$ is a constant which depends only on the objects in the list and the dimension $n$. Furthermore, $A \sim_L B$ signifies that $A \lesssim_L B$ and $B \lesssim_L A$. 
\end{notation}

\begin{acknowledgement} The author is indebted to Christopher D. Sogge both for suggesting the problem and for providing a number of helpful comments regarding the presentation. This research was partly carried out during a visit to the Institute of Applied Physics and Computational Mathematics, Beijing, in June 2019 and the author would like to thank Changxing Miao for his kind hospitality. 
\end{acknowledgement}



\section{Spectral projections}\label{spectral projector section} 




\subsection*{An equivalent formulation} It was shown in~\cite{BSSY} that the desired resolvent estimates are equivalent to certain spectral projection bounds. Given $\lambda \geq 1$ and $\rho > 0$, define 
\begin{equation*}
    A(\lambda, \rho) := \big\{ \xi \in \hat{\R}^n : \big||\xi| - \lambda\big| < \rho \big\}.
\end{equation*}
In the case of the torus,~\cite[Theorem 1.3]{BSSY} implies the following. 

\begin{theorem}[\cite{BSSY}]\label{equivalence theorem} Given $n \geq 3$ and $0< \alpha \leq 1$, the following are equivalent:
\begin{enumerate}[i)]
    \item For all $\lambda \geq 1$ there is a uniform spectral projection estimate
    \begin{equation}\label{uniform projection equation}
        \Big\| \sum_{k \in \Z^n \cap A(\lambda, \lambda^{-\alpha})} \hat{f}(k) e^{2 \pi i x \cdot k} \Big\|_{L^{\frac{2n}{n-2}}(\T^n)} \lesssim_{\alpha} \lambda^{1 -\alpha} \|f\|_{L^{\frac{2n}{n+2}}(\T^n)}.
    \end{equation}
    \item There is a uniform resolvent estimate
    \begin{equation*}
          \|u\|_{L^{\frac{2n}{n-2}}(\T^n)} \lesssim_{\alpha} \| (\Delta_{\T^n} + z)u\|_{L^{\frac{2n}{n+2}}(\T^n)}
    \end{equation*}
    for all $z = (\lambda + i \mu)^2 \in \C$ such that $\lambda, \mu \in \R$ satisfy $\lambda \geq 1$, $|\mu| \geq \lambda^{-\alpha}$.
\end{enumerate}
\end{theorem}

\begin{remark} In~\cite{BSSY} a more general statement is proven for compact manifolds.
\end{remark}

The remaining sections of this paper will focus on proving spectral projection bounds of the type featured above. 




\subsection*{Relationship with discrete Fourier restriction} Although it will not play any r\^ole in later arguments, it is nevertheless instructive to remark that Theorem~\ref{equivalence theorem} relates the resolvent and discrete restriction conjectures. 

\begin{conjecture}[Discrete restriction conjecture~\cite{Bourgain1993}]\label{discrete restriction conjecture} For $n \geq 3$, $\lambda \geq 1$ and $\varepsilon > 0$,
\begin{equation}\label{discrete restriction equation}
        \Big\| \sum_{k \in \Z^n \cap \lambda S^{n-1}} \hat{f}(k) e^{2 \pi i x \cdot k} \Big\|_{L^{\frac{2n}{n-2}}(\T^n)} \lesssim_{\varepsilon} \lambda^{\varepsilon} \|f\|_{L^{2}(\T^n)}.
    \end{equation}
\end{conjecture}

In particular, if $e_{\lambda}$ is an $L^2$-normalised eigenfunction for $-\Delta_{\T}$ with eigenvalue $\lambda^2$, then Conjecture~\ref{discrete restriction conjecture} implies that $\|e_{\lambda}\|_{L^{2n/(n-2)}(\T^n)} \lesssim_{\varepsilon} \lambda^{\varepsilon}$. Various partial results on this problem are known, establishing weaker versions of \eqref{discrete restriction equation} with larger values of $p$ on the left-hand side: see~\cite{Bourgain1993, Bourgain1997, BD2013,BD2015a, BD2015}.

By elementary separation properties of concentric lattice spheres, \eqref{discrete restriction equation} is equivalent to 
 \begin{equation*}
        \Big\| \sum_{k \in \Z^n \cap A(\lambda, \lambda^{-1})} \hat{f}(k) e^{2 \pi i x \cdot k} \Big\|_{L^{\frac{2n}{n-2}}(\T^n)} \lesssim_{\varepsilon} \lambda^{\varepsilon} \|f\|_{L^{2}(\T^n)}.
    \end{equation*}
It is not difficult to see, using a $T^*T$ argument, that the above estimate would follow from \eqref{uniform projection equation} with $\alpha = 1 - \varepsilon$. Thus, by Theorem~\ref{equivalence theorem}, the resolvent conjecture (Conjecture~\ref{resolvent conjecture}) implies the discrete restriction conjecture (Conjecture~\ref{discrete restriction conjecture}).




\section{The proof of Theorem~\ref{resolvent theorem}}\label{proof section}

By Theorem~\ref{equivalence theorem}, the uniform resolvent estimates in Theorem~\ref{resolvent theorem} are equivalent to the following spectral projection bounds. 

\begin{proposition}\label{spectral proposition} Let $n \geq 3$, $\lambda \geq 1$ and $\varepsilon >0$. If $\rho := \lambda^{-1/3 + \varepsilon}$, then
\begin{equation}\label{spectral inequality}
    \Big\| \sum_{k \in \Z^n \cap A(\lambda, \rho)} \hat{f}(k) e^{2 \pi i x \cdot k} \Big\|_{L^{\frac{2n}{n-2}}(\T^n)} \lesssim_{\varepsilon} \rho \lambda \|f\|_{L^{\frac{2n}{n+2}}(\T^n)}.
\end{equation}
\end{proposition}

Given $m \in \ell^{\infty}(\Z^n)$ let $m(D)$ denote the associated Fourier multiplier operator, defined initially on $C^{\infty}(\T^n)$ by 
\begin{equation*}
    m(D)f(x) := \sum_{k \in \Z^n}  m(k) \hat{f}(k)e^{2 \pi i x \cdot k}.
\end{equation*}
If $m \in L^{\infty}(\hat{\R}^n)$, then $m(D) := m|_{\Z^n}(D)$ where $m|_{\Z^n}$ denotes the restriction of $m$ to the integer lattice. Thus, with this notation, one may write~\eqref{spectral inequality} as
\begin{equation}\label{spectral inequality 2}
    \| \chi_{A(\lambda, \rho)}(D) f\|_{L^{\frac{2n}{n-2}}(\T^n)} \lesssim_{\varepsilon} \rho \lambda  \|f\|_{L^{\frac{2n}{n+2}}(\T^n)}.
\end{equation}

The remainder of this section deals with the proof of Proposition~\ref{spectral proposition}.




\subsection*{Smooth multipliers} In proving Proposition~\ref{spectral proposition}, one may replace the rough cutoff function $\chi_{A(\lambda, \rho)}$ with a smoothed out version. Indeed, by $T^*T$,~\eqref{spectral inequality 2} is equivalent to
\begin{equation}\label{spectral 1}
      \| \chi_{A(\lambda, \rho)}(D)f\|_{L^{2}(\T^n)} \lesssim_{\varepsilon} (\rho\lambda)^{1/2} \|f\|_{L^{\frac{2n}{n+2}}(\T^n)}.
\end{equation}
Fix $\beta \in C_c^{\infty}(\R)$ non-negative with $\beta(r) = 1$ for $|r| \leq 1$ and $\beta(r) = 0$ for $|r| \geq 2$ and define the multiplier
\begin{equation}\label{smooth multiplier}
    m^{\lambda,\rho}(\xi) := \beta\big(\rho^{-1}(|\xi| - \lambda)\big).
\end{equation}
By $L^2$-orthogonality,~\eqref{spectral 1} would follow from the bound
\begin{equation*}
      \big\|m^{\lambda,\rho}(D)^{1/2} f \big\|_{L^{2}(\T^n)} \lesssim_{\varepsilon}  (\rho\lambda)^{1/2} \|f\|_{L^{\frac{2n}{n+2}}(\T^n)}
\end{equation*}
and, by a second application of $T^*T$, this would further follow from
\begin{equation}\label{spectral 2} 
\| m^{\lambda,\rho}(D) f\|_{L^{\frac{2n}{n-2}}(\T^n)} \lesssim_{\varepsilon} \rho \lambda  \|f\|_{L^{\frac{2n}{n+2}}(\T^n)}.
\end{equation}




\subsection*{Consequences of $\ell^2$-decoupling} The proof of Proposition~\ref{spectral proposition} relies on the $\ell^2$-decoupling theorem proved in~\cite{BD2015}. It is convenient to work with a rescaled version of the decoupling theorem, in the special case of the euclidean sphere. For $\lambda \geq 1$ and $g \in L^1(\lambda S^{n-1})$ let
\begin{equation*}
    E_{\lambda} g(x) := \int_{\lambda S^{n-1}} g(\omega) e^{2 \pi i x \cdot \omega} \,\ud \sigma_{\lambda S^{n-1}}(\omega), \qquad x \in \R^n,
\end{equation*}
where the integration is with respect to the normalised (to have unit mass) surface measure on $\lambda S^{n-1}$.

\begin{theorem}[Bourgain--Demeter~\cite{BD2015}]\label{decoupling theorem} Let $\lambda \gtrsim 1$, $1 \gtrsim \rho \geq \lambda^{-1}$ and $\Theta(\lambda, \rho)$ be a finitely-overlapping covering of $\lambda S^{n-1}$ by $(\rho\lambda)^{1/2}$-caps. Given $g \in L^1(S^{n-1})$ write $g_{\theta} := g \cdot \chi_{\theta}$. For all $\varepsilon > 0$,
\begin{equation*}
    \|E_{\lambda} g\|_{L^{\frac{2(n+1)}{n-1}}(B_{\rho^{-1}})} \lesssim_{\varepsilon} \lambda^{\varepsilon} \Big(\sum_{\theta \in \Theta(\lambda,\rho)} \|E_{\lambda} g_{\theta}\|_{L^{\frac{2(n+1)}{n-1}}(w_{B_{\rho^{-1}}})}^2\Big)^{\frac{1}{2}}.
\end{equation*}
\end{theorem}

Here $B_r$ is used to denote an \emph{$r$-ball}: that is, $B_r$ is a ball in $\R^n$ with (arbitrary) centre $c(B_r)$ and radius $r > 0$. The weight $w_{B_r}$ is the function concentrated on $B_r$ given by
\begin{equation}\label{weight}
    w_{B_r}(x) := \big(1 + |x - c(B_r)|)^{-6N}
\end{equation}
where $N := 100n$. Finally, an \emph{$r$-cap} on the sphere $\lambda S^{n-1}$ is the intersection of $\lambda S^{n-1}$ with an $r$-ball centred at a point on $\lambda S^{n-1}$. 

Using Theorem~\ref{decoupling theorem}, one may prove an $L^{\frac{2(n+1)}{n+3}}\to L^{\frac{2(n+1)}{n-1}}$ bound for the projector in~\eqref{spectral inequality}. 

\begin{corollary}\label{decoupling corollary 1} Let $n \geq 3$, $\lambda \gtrsim 1$ and $1 \gtrsim r \geq \lambda^{-1}$. For all $\varepsilon > 0$,
\begin{equation}\label{decoupling corollary equation}
    \big\| \chi_{A(\lambda, r)}(D)f \big\|_{L^{\frac{2(n+1)}{n-1}}(\T^n)} \lesssim_{\varepsilon} \lambda^{\varepsilon} (r\lambda)^{\frac{n-1}{n+1}} \|f\|_{L^{\frac{2(n+1)}{n+3}}(\T^n)}.
\end{equation}
\end{corollary}
By duality and $T^*T$,~\eqref{decoupling corollary equation} is equivalent to either of the following inequalities:
\begin{align}
\label{corollary 1}
     \|\chi_{A(\lambda, r)}(D)f\|_{L^{\frac{2(n+1)}{n-1}}(\T^n)} &\lesssim_{\varepsilon} \lambda^{\varepsilon} (r\lambda)^{\frac{n-1}{2(n+1)}} \|f\|_{L^{2}(\T^n)}, \\
\label{corollary 0}
    \|\chi_{A(\lambda,r)}(D)f\|_{L^2(\T^n)} &\lesssim_{\varepsilon} \lambda^{\varepsilon} (r\lambda)^{\frac{n-1}{2(n+1)}}\|f\|_{L^{\frac{2(n+1)}{n+3}}(\T^n)}.
\end{align}
\begin{remark} If $r = \lambda^{-1}$, then Corollary~\ref{decoupling corollary 1} corresponds to a special case of the discrete Fourier restriction theorem of Bourgain--Demeter~\cite[Theorem 2.2]{BD2015}. On the other hand, if $r \sim 1$, then~\eqref{decoupling corollary equation} holds with no $\varepsilon$-loss as a simple consequence of the Stein--Tomas restriction theorem for the sphere, as discussed below. 
\end{remark}

\begin{proof}[Proof (of Corollary~\ref{decoupling corollary 1})] As remarked earlier, it suffices to prove~\eqref{corollary 1}. It is well known (see, for instance,~\cite{BD2015}) that Theorem~\ref{decoupling theorem} implies a discrete version of itself. In particular, defining $R := r^{-1}$, given any 1-separated subset $\Omega_{\lambda} \subseteq \lambda S^{n-1}$ and any sequence $(a_{\omega})_{\omega \in \Omega_{\lambda}}$, it follows that  
\begin{equation}\label{discrete decoupling 1}
       \Big\|\sum_{\omega \in \Omega_{\lambda}} a_{\omega} e^{2 \pi i x \cdot \omega}\Big\|_{L^{\frac{2(n+1)}{n-1}}(B_{R})} \lesssim_{\varepsilon} \lambda^{\varepsilon} \Big(\sum_{\theta \in \Theta(\lambda, r)} \Big\|\sum_{\omega \in \Omega_{\lambda} \cap \theta} a_{\omega} e^{2 \pi i x \cdot \omega}\Big\|_{L^{\frac{2(n+1)}{n-1}}(w_{B_{R}})}^2\Big)^{\frac{1}{2}}.
\end{equation}
Indeed, this may be deduced by fixing $\psi \in C^{\infty}_c(\hat{\R}^n)$ with $\psi(0) = 1$, applying Theorem~\ref{decoupling theorem} to the functions \begin{equation*}
   g_{\delta}(w) := \sum_{\omega \in \Omega_{\lambda}} a_{\omega} \psi(\delta^{-1}(w - \omega)) 
\end{equation*}
for $\delta > 0$ and applying a simple limiting argument; see~\cite{BD2015}. 

The spatial variable in~\eqref{discrete decoupling 1} is localised to a ball of radius $R = r^{-1}$, inducing frequency uncertainty at scale $r$. In particular, one can (at least heuristically) replace the family of points $\Omega_{\lambda}$ in this inequality with any perturbed family
\begin{equation*}
    \tilde{\Omega}_{\lambda} = \{\omega + O(r) : \omega \in \Omega_{\lambda}\}.
\end{equation*} 
For instance, one may take $\tilde{\Omega}_{\lambda} := \Z^n \cap A(\lambda,r)$, in which case~\eqref{discrete decoupling 1} implies that
\begin{equation}\label{discrete decoupling 1.5}
        \|\chi_{A(\lambda, r)}(D) f\|_{L^{\frac{2(n+1)}{n-1}}(B_{R})} \lesssim_{\varepsilon} \lambda^{\varepsilon} \Big(\sum_{\theta \in \Theta(\lambda, r)} \|\chi_{A_{\theta}(\lambda, r)}(D)f\|_{L^{\frac{2(n+1)}{n-1}}(w_{B_{R}})}^2\Big)^{1/2}
\end{equation}
where $A_{\theta}(\lambda,r)$ is the intersection of $A(\lambda, r)$ with the sector generated by $\theta$. Giving a rigorous justification for this uncertainty heuristic is a messy affair and is therefore postponed until the end of the proof.

Since the functions appearing in either side of~\eqref{discrete decoupling 1.5} are $1$-periodic, it follows that
\begin{equation*}
        \|\chi_{A(\lambda, r)}(D) f\|_{L^{\frac{2(n+1)}{n-1}}(\T^n)} \lesssim_{\varepsilon} \lambda^{\varepsilon} \Big(\sum_{\theta \in \Theta(\lambda, r)} \|\chi_{A_{\theta}(\lambda, r)}(D) f\|_{L^{\frac{2(n+1)}{n-1}}(\T^n)}^2\Big)^{1/2}.
\end{equation*}
To bound the right-hand side, observe the elementary estimate
\begin{equation*}
    \|\chi_{A_{\theta}(\lambda;r)}(D)f\|_{L^{\infty}(\T^n)} \leq [\#\Z^n \cap A_{\theta}(\lambda,r)]^{1/2} \|\chi_{A_{\theta}(\lambda,r)}(D)f\|_{L^2(\T^n)} 
    \end{equation*}
holds by a combination of Cauchy--Schwarz and Plancherel's theorem. Thus, given $2 \leq p \leq \infty$, it follows that
\begin{equation}\label{discrete decoupling 2}
    \|\chi_{A_{\theta}(\lambda,r)}(D)f \|_{L^{p}(\T^n)} \leq [\#\Z^n \cap A_{\theta}(\lambda,r)]^{1/2 - 1/p} \|\chi_{A_{\theta}(\lambda,r)}(D)f\|_{L^2(\T^n)}.
\end{equation}
Applying the bound $\#\Z^n \cap A_{\theta}(\lambda, r) \lesssim (r\lambda)^{\frac{n-1}{2}}$, taking $\ell^2$-norms in $\theta$ of both sides of the above inequality and using Plancherel's theorem to sum, the desired estimate follows. 

It remains to give a rigorous justification of the uncertainty principle heuristic used in the above argument. Given $k \in \Z^n \cap A(\lambda, r)$ let $\omega_k$ denote the point on $\lambda S^{n-1}$ closest to $k$, so that $|\omega_k - k| < r$, and $\Omega_{\lambda}$ denote the collection of all such $\omega_k$. Suppose $\bar{x} \in \R^n$ is the centre of $B_{R}$. Applying the Taylor series expansion for the exponential,
\begin{align*}
    \chi_{A(\lambda,r)}(D)f(x) &=  \sum_{\alpha \in \N_0^n} \frac{(2\pi i)^{|\alpha|}(x - \bar{x})^{\alpha}}{\alpha!}\sum_{\omega_k \in \Omega_{\lambda}} (k-\omega_k)^{\alpha} \hat{f}(k)e^{2 \pi i \bar{x} \cdot(k- \omega_k)}  e^{2 \pi i x \cdot \omega_k} \\
    &=:  \sum_{\alpha \in \N_0^n} \frac{(2\pi i)^{|\alpha|}(x - \bar{x})^{\alpha}}{\alpha!}\sum_{\omega \in \Omega_{\lambda}} a_{\alpha,\omega} e^{2 \pi i x \cdot \omega} ,
\end{align*}
where $|\alpha| = \alpha_1 + \cdots + \alpha_n$, $\alpha! = \alpha_1! \cdots \alpha_n!$ and $x^{\alpha} = x_1^{\alpha_1} \cdots x_n^{\alpha_n}$ for $\alpha \in \N_0^n$ and $x \in \R^n$. Thus, by the triangle inequality and~\eqref{discrete decoupling 1}, the left-hand side of~\eqref{discrete decoupling 1.5} is dominated by
\begin{equation*}
  \lambda^{\varepsilon} \sum_{\alpha \in \N_0^n} \frac{(2\pi R)^{|\alpha|}}{\alpha!} \Big(\sum_{\theta \in \Theta(\lambda, r)} \Big\|\sum_{\omega \in \Omega_{\lambda} \cap \theta} a_{\alpha,\omega} e^{2 \pi i x \cdot \omega}\Big\|_{L^{\frac{2(n+1)}{n-1}}(w_{B_{R}})}^2\Big)^{\frac{1}{2}}.
\end{equation*}
Given $l \in \Z^n$ write $\bar{x}_l := R l$ and $B^l := B( \bar{x}_l, R)$ so that
\begin{equation*}
    \Big\|\sum_{\omega \in \Omega_{\lambda} \cap \theta} a_{\alpha,\omega} e^{2 \pi i x \cdot \omega}\Big\|_{L^{\frac{2(n+1)}{n-1}}(w_{B_{R}})} \lesssim \sum_{l \in \Z^n} (1 + |l|)^{-N} \Big\|\sum_{\omega \in \Omega_{\lambda} \cap \theta} a_{\alpha,\omega} e^{2 \pi i x \cdot \omega}\Big\|_{L^{\frac{2(n+1)}{n-1}}(B^l)},
\end{equation*}
where $N := 100n$ is the exponent appearing in the definition of the weight function from \eqref{weight}. Indeed, this follows by pointwise dominating $w_{B_{R}}$ by a weighted sum of characteristic functions. As before, one may write
\begin{align*}
\sum_{\omega \in \Omega_{\lambda} \cap \theta} a_{\alpha,\omega} e^{2 \pi i x \cdot \omega} &=  \sum_{\beta \in \N_0^n} \frac{(2\pi i)^{|\beta|}(x - \bar{x}_l)^{\beta}}{\beta!}\sum_{\omega_k \in \Omega_{\lambda}} (\omega_k-k)^{\beta} a_{\alpha,\omega_k}e^{2 \pi i \bar{x}_l \cdot(\omega_k-k)}  e^{2 \pi i x \cdot k} \\
    &=  \sum_{\beta \in \N_0^n} \frac{(2\pi i)^{|\beta|}(x - \bar{x}_l)^{\beta}}{\beta!} \chi_{A_{\theta}(\lambda,r)}(D)m_{\alpha,\beta,l}(D) f(x)
\end{align*}
where $m_{\alpha, \beta, l}$ is supported on $\Z^n \cap A(\lambda, r)$ and is given by
\begin{equation*}
    m_{\alpha,\beta,l}(k) := (-1)^{|\beta|}(k-\omega_k)^{\alpha + \beta} e^{2 \pi i (\bar{x}-\bar{x}_l) \cdot (k - \omega_k)}  \quad \textrm{for $k \in \Z^n \cap A(\lambda,r)$.} 
\end{equation*}
In particular, 
\begin{equation*}
    \max_{k \in \Z^n \cap A(\lambda,r)} |m_{\alpha,\beta,l}(k)| \lesssim r^{|\alpha| + |\beta|}.
\end{equation*}

By combining the above observations, applying the triangle inequality and exploiting periodicity, one concludes that $\|\chi_{A(\lambda, r)}(D) f\|_{L^{\frac{2(n+1)}{n-1}}(\T^n)}$ is dominated by
\begin{equation*}
   \lambda^{\varepsilon} \sum_{\substack{\alpha,\beta \in \N_0^n \\ l \in \Z^n}} \frac{(2\pi R)^{|\alpha|+|\beta|}}{\alpha!\beta!}  (1 + |l|)^{-N}  \Big(\sum_{\theta \in \Theta(\lambda, r)} \|\chi_{A_{\theta}(\lambda, r)}(D) m_{\alpha,\beta,l}(D) f\|_{L^{\frac{2(n+1)}{n-1}}(\T^n)}^2\Big)^{1/2}.
\end{equation*}
Finally, a slight modification of the argument used to prove~\eqref{discrete decoupling 2} shows that, given $2 \leq p \leq \infty$,
\begin{equation*}
   \|\chi_{A_{\theta}(\lambda,r)}(D)m_{\alpha, \beta,l}(D)f \|_{L^{p}(\T^n)} \lesssim r^{|\alpha| + |\beta|} [\#A_{\theta}(\lambda,r)]^{1/2 - 1/p} \|\chi_{A_{\theta}(\lambda,r)}(D)f\|_{L^2(\T^n)}.
\end{equation*}
The gain in $r$ in the previous inequality compensates for the earlier losses in $R$ and the desired estimate now readily follows from Plancherel's theorem.
\end{proof}

\begin{corollary}\label{decoupling corollary 2} Let $n \geq 3$, $\lambda \geq 1$ and $1 \gtrsim r > \lambda^{-1}$ and suppose $m \in \ell^{\infty}(\Z^n)$ is supported in $A(\lambda,r)$. For all $\varepsilon > 0$,
\begin{equation*}
    \| m(D)f \|_{L^{\frac{2(n+1)}{n-1}}(\T^n)} \lesssim_{\varepsilon} \lambda^{\varepsilon} (r\lambda)^{\frac{n-1}{n+1}} \|m\|_{\ell^{\infty}(\Z^n)}\|f\|_{L^{\frac{2(n+1)}{n+3}}(\T^n)}.
\end{equation*}
\end{corollary}

\begin{proof} The corollary follows easily by writing 
\begin{equation*}
  m = \chi_{A(\lambda,r)}\cdot m \cdot \chi_{A(\lambda,r)}   
\end{equation*}
and successively applying~\eqref{corollary 1}, Plancherel's theorem and~\eqref{corollary 0}.
\end{proof}




\subsection*{Consequences of the Stein--Tomas theorem} An equivalent formulation of the Stein--Tomas restriction theorem for the sphere is that
\begin{equation}\label{Stein--Tomas}
   \Big( \int_{A(\lambda,1)} |\hat{F}(\xi)|^2 \,\ud \xi \Big)^{1/2} \lesssim \lambda^{\frac{n-1}{2(n+1)}} \|F\|_{L^{p_0'}(\R^n)};
\end{equation}
see, for instance,~\cite{Tao2004} or~\cite[Chapter 5]{Sogge2017}. This implies a version of Corollary \ref{decoupling corollary 1} for $r = 1$ with no $\varepsilon$-loss in the exponent.
\begin{corollary}\label{Stein--Tomas corollary 1} Let $n \geq 3$ and $\lambda \gtrsim 1$. Then
\begin{equation*}
    \big\| \chi_{A(\lambda, 1)}(D)f \big\|_{L^{\frac{2(n+1)}{n-1}}(\T^n)} \lesssim \lambda^{\frac{n-1}{n+1}} \|f\|_{L^{\frac{2(n+1)}{n+3}}(\T^n)}.
\end{equation*}
\end{corollary}

 \begin{remark} Corollary \ref{Stein--Tomas corollary 1} is also a special case of a more general spectral projection bound for compact Riemann manifolds: see~\cite{Sogge1988} or~\cite[Chapter 5]{Sogge2017}. 
 \end{remark}

\begin{proof}[Proof (of Corollary \ref{Stein--Tomas corollary 1})] As before, by $T^*T$ the desired estimate is equivalent to 
\begin{equation}\label{Stein--Tomas 1}
  \big\|\chi_{A(\lambda,1)}(D)f\big\|_{L^2(\T^n)} \lesssim \lambda^{\frac{n-1}{2(n+1)}}\|f\|_{L^{\frac{2(n+1)}{n+3}}(\T^n)}.   
\end{equation}

Fix $f \in C^{\infty}(\T^n)$ and let $\psi \in \mathcal{S}(\R^n)$ be non-zero and Fourier supported in a ball of radius $1/2$. Letting $F \in \mathcal{S}(\R^n)$ be defined by
\begin{equation*}
    F(x) := \sum_{k \in \Z^n} \hat{f}(k) e^{2 \pi i x \cdot k} \psi(x),
\end{equation*}
the estimate~\eqref{Stein--Tomas 1} now follows by applying~\eqref{Stein--Tomas} to this function. 
\end{proof}

Arguing precisely as in the previous subsection, Corollary \ref{Stein--Tomas corollary 1} implies a version of Corollary \ref{decoupling corollary 2} for $r=1$ with no $\varepsilon$-loss.

\begin{corollary}\label{Stein--Tomas corollary 2} Let $n \geq 3$ and  $\lambda \geq 1$ and suppose $m \in \ell^{\infty}(\Z^n)$ is supported in $A(\lambda,1)$. Then
\begin{equation*}
    \| m(D)f \|_{L^{\frac{2(n+1)}{n-1}}(\T^n)} \lesssim \lambda^{\frac{n-1}{n+1}} \|m\|_{\ell^{\infty}(\Z^n)}\|f\|_{L^{\frac{2(n+1)}{n+3}}(\T^n)}.
\end{equation*}
\end{corollary}

\begin{remark} Corollary \ref{Stein--Tomas corollary 2} is also a special instance of the multiplier lemma from~\cite[Lemma 2.3]{BSSY}, which applies to more general compact Riemannian manifolds.
\end{remark}




\subsection*{Proof of the spectral projection bound} The ingredients introduced above may now be combined to prove the desired spectral projection bound. 

\begin{proof}[Proof (of Proposition~\ref{spectral proposition})] Fixing $\varepsilon > 0$, recall that it suffices to show~\eqref{spectral 2} holds for $\rho = \lambda^{-1/3 + \varepsilon}$. In order to justify this choice of $\rho$, and in view of the proof of Theorem~\ref{refined resolvent theorem} below, it will be convenient to initially let $\rho$ denote some unspecified parameter satisfying $1 \gtrsim \rho \geq \lambda^{-1}$ and only fix the value later in the argument.

Fix a Schwartz function $\eta$ on $\hat{\R}^n$ satisfying $\check{\eta}(x) = 1$ whenever $|x| \leq 1$. Recalling the definition of the smoothed out multiplier $m^{\lambda, \rho}$ from \eqref{smooth multiplier}, decompose
 \begin{equation*}
    m^{\lambda, \rho} = m_0^{\lambda, \rho} + m_1^{\lambda, \rho} 
\end{equation*}
where $m_0^{\lambda,\rho} := m^{\lambda,\rho} \ast \eta$. Writing $p_1 := \frac{2n}{n-2}$, it follows that
\begin{equation}\label{spectral proof 1}
    \|m^{\lambda, \rho}(D)\|_{p_1' \to p_1} \leq \|m_0^{\lambda, \rho}(D)\|_{p_1' \to p_1} + \|m_1^{\lambda, \rho}(D)\|_{p_1' \to p_1}
\end{equation}
where $p'$ denotes the H\"older conjugate of a Lebesgue exponent $p$.

Both term on the right-hand side of~\eqref{spectral proof 1} are estimated via complex interpolation between an $L^{p_0'} \to L^{p_0}$ bound for $p_0 := \frac{2(n+1)}{n-1}$ and an $L^{1} \to L^{\infty}$ bound. In particular, by the Riesz--Thorin theorem,
\begin{equation}\label{spectral proof 2}
    \|m_j^{\lambda, \rho}(D)\|_{p_1' \to p_1} \leq \|m_j^{\lambda, \rho}(D)\|_{p_0' \to p_0}^{\frac{(n-2)(n+1)}{n(n-1)}} \|m_j^{\lambda, \rho}(D)\|_{1 \to \infty}^{\frac{2}{n(n-1)}} \quad \textrm{for $j = 0,1$.}
\end{equation}

To bound $m_0^{\lambda, \rho}(D)$, apply a partition of unity to decompose
\begin{equation*}
    \eta = \sum_{\ell \in \Z^n} (1+|\ell|)^{-N} \tilde{\eta}_{\ell}
\end{equation*}
where $N := 100n$ and each $\tilde{\eta}_{\ell}$ is supported on the ball of unit radius centred at $\ell$ and satisfies $\|\tilde{\eta}_{\ell}\|_{\infty} \lesssim 1$. Note that the latter property holds due to the rapid decay of $\eta$. This induces a corresponding decomposition of the multiplier
\begin{equation}\label{spectral proof 2.5}
    m_0^{\lambda, \rho} = \sum_{\ell \in \Z^n} (1+|\ell|)^{-N} \tilde{m}_{\ell}^{\lambda, \rho}
\end{equation}
where each $\tilde{m}_{\ell}^{\lambda, \rho}$ is supported on the Minkowski sum \begin{equation*}
    \mathrm{supp}\,m^{\lambda,\rho} + \mathrm{supp}\,\tilde{\eta}_{\ell} \subseteq \ell + A(\lambda, 4).
\end{equation*} Furthermore,
\begin{equation}\label{spectral proof 3}
    \|\tilde{m}_{\ell}^{\lambda,\rho}\|_{\ell^{\infty}(\Z^n)} \lesssim \rho \quad \textrm{and} \quad \|\tilde{m}_{\ell}^{\lambda,\rho}\|_{\ell^{1}(\Z^n)} \lesssim \rho \lambda^{n-1}.
    \end{equation}
To see this, observe that $|\tilde{m}_{\ell}^{\lambda,\rho}(\xi)| \lesssim |B(\ell + \xi,1) \cap A(\lambda, \rho)|$, which immediately yields the $\ell^{\infty}$ estimate. The $\ell^1$ bound then follows from the $\ell^{\infty}$ estimate and the fact that $\#\big(\Z^n \cap \ell + A(\lambda,4)\big) \lesssim \lambda^{n-1}$. Consequently, and in view of Corollary~\ref{Stein--Tomas corollary 2},
\begin{equation}\label{spectral proof 3.5}
    \|\tilde{m}_{\ell}^{\lambda,\rho}(D)\|_{p_0' \to p_0} \lesssim \rho\lambda^{\frac{n-1}{n+1}} \quad \textrm{and} \quad  \|\tilde{m}_{\ell}^{\lambda,\rho}(D)\|_{1 \to \infty} \lesssim \rho \lambda^{n-1}.
\end{equation}
More precisely, the first inequality in \eqref{spectral proof 3.5} follows from Corollary~\ref{Stein--Tomas corollary 2} together with the $\ell^{\infty}$ estimate from~\eqref{spectral proof 3}. Here it is important to use Corollary~\ref{Stein--Tomas corollary 2} rather than Corollary~\ref{decoupling corollary 2} to ensure that there is no $\varepsilon$-loss in the exponent. The second inequality in \eqref{spectral proof 3.5} is a direct consequence of the $\ell^1$ estimate in~\eqref{spectral proof 3} (which allows one to bound the $\ell^{\infty}$ norm of the kernel associated to $m_{\ell}^{\lambda,\rho}(D)$).

Using the triangle inequality and the decay factor in \eqref{spectral proof 2.5} to sum the above estimates,
\begin{equation}\label{spectral proof 4}
    \|m_{0}^{\lambda,\rho}(D)\|_{p_0' \to p_0} \lesssim \rho\lambda^{\frac{n-1}{n+1}} \quad \textrm{and} \quad  \|m_{0}^{\lambda,\rho}(D)\|_{1 \to \infty} \lesssim \rho \lambda^{n-1}.
\end{equation}
Interpolating the two inequalities in~\eqref{spectral proof 4} via~\eqref{spectral proof 2}, one deduces that
\begin{equation}\label{spectral proof 5}
   \|m_{0}^{\lambda, \rho}(D)\|_{p_1' \to p_1}\lesssim  \rho\lambda.
\end{equation}

It remains to bound $m_1^{\lambda, \rho}(D)$. Since the multiplier $m^{\lambda,\rho}$ is supported in $A(\lambda, 2\rho)$ and is uniformly bounded, it follows from Corollary~\ref{decoupling corollary 2} that
\begin{equation}\label{spectral proof 6}
    \|m_1^{\lambda, \rho}(D)\|_{p_0' \to p_0} \leq \|m_0^{\lambda, \rho}(D)\|_{p_0' \to p_0} + \|m^{\lambda, \rho}(D)\|_{p_0' \to p_0} \lesssim_{\varepsilon} \lambda^{\varepsilon} (\rho\lambda)^{\frac{n-1}{n+1}},
\end{equation}
where the first term on the right-hand side is estimated using~\eqref{spectral proof 4}. On the other hand, it is claimed that
\begin{equation}\label{spectral proof 7}
\|m_1^{\lambda, \rho}(D)\|_{1 \to \infty} \lesssim (\lambda/\rho)^{(n-1)/2}.
\end{equation}
Temporarily assuming this bound, interpolating~\eqref{spectral proof 7} against~\eqref{spectral proof 6} via~\eqref{spectral proof 2} yields
\begin{equation}\label{spectral proof 8}
   \|m_1^{\lambda, \rho}(D)\|_{p_1' \to p_1} \lesssim_{\varepsilon} \lambda^{\varepsilon}  \rho^{1-3/n}\lambda^{1-1/n}.
\end{equation}
Substituting~\eqref{spectral proof 5} and~\eqref{spectral proof 8} into~\eqref{spectral proof 1}, one concludes that
\begin{equation}\label{spectral proof 9}
    \|m^{\lambda,\rho}(D)\|_{p_1' \to p_1} \lesssim_{\varepsilon} \rho\lambda +  \lambda^{\varepsilon}\rho^{1-3/n} \lambda^{1-1/n}.
\end{equation}
Replacing $\varepsilon$ with $3\varepsilon/n$ in the above display and choosing $\rho = \lambda^{-1/3 + \varepsilon}$ so as to optimise the estimate, one deduces the desired bound. Thus, it remains to verify~\eqref{spectral proof 7}. 

Computing the kernel of $m^{\lambda,\rho}_1(D)$ and applying the Poisson summation formula,
\begin{equation}\label{spectral proof 10}
   \|m_1^{\lambda, \rho}(D)\|_{1 \to \infty} \leq \sup_{x \in \T^n} \Big|\sum_{k \in \Z^n}  m^{\lambda,\rho}_1(k)e^{2 \pi i x \cdot k} \Big| = \sup_{x \in \T^n} \Big|\sum_{k \in \Z^n} \big(m^{\lambda,\rho}_1\big)\;\widecheck{}\;(x+k) \Big|.  
\end{equation}
Note that $\big(m^{\lambda,\rho}_1\big)\;\widecheck{}\;(x) = \big(m^{\lambda,\rho}\big)\;\widecheck{}\;(x)\big(1- \check{\eta}(x)\big)$. If $\sigma$ denotes the surface measure on $S^{n-1}$, then applying polar coordinates to the definition of the Fourier transform yields
\begin{equation}\label{spectral proof 11}
    \big(m^{\lambda,\rho}\big)\;\widecheck{}\;(x) = \int_0^{\infty} \check{\sigma}(rx) \beta\big(\rho^{-1}(r - \lambda)\big) r^{n-1}\,\ud r.
\end{equation}
By stationary phase (see, for instance,~\cite[Chapter VIII]{Stein1993} or~\cite[Chapter 1]{Sogge2017}), 
\begin{equation*}
    \check{\sigma}(x) = \sum_{\pm} e^{\pm 2 \pi i |x|} a_{\pm}(x)
\end{equation*}
where each $a_{\pm} \in C^{\infty}(\R^n)$ is a symbol of order $-(n-1)/2$ in the sense that $|\partial_x^{\alpha} a_{\pm}(x)| \lesssim_{\alpha} (1+|x|)^{-(n-1)/2 - |\alpha|}$ for all $\alpha \in \N_0^n$. Substituting this identity into~\eqref{spectral proof 11} and applying a change of variables,
\begin{equation}\label{spectral proof 12}
    \big(m^{\lambda,\rho}\big)\;\widecheck{}\;(x) =  \rho \sum_{\pm}\int_0^{\infty} e^{\pm 2 \pi i r\rho|x|} a_{\pm}(\rho rx) \beta\big(r - \rho^{-1}\lambda\big) (\rho r)^{n-1}\,\ud r.
\end{equation}
Applying repeated integration-by-parts, it follows that
\begin{equation}\label{spectral proof 13}
    |\big(m^{\lambda,\rho}\big)\;\widecheck{}\;(x)| \lesssim \rho \lambda^{n-1} (1+\lambda|x|)^{-(n-1)/2} (1 + \rho|x|)^{-N} \lesssim \frac{\rho \lambda^{(n-1)/2}}{|x|^{(n-1)/2}} (1 + \rho|x|)^{-N}.
\end{equation}

To bound the right-hand side of~\eqref{spectral proof 10} the sum is broken into two pieces. Fix $x \in \T^n$ and write
\begin{equation*}
    \Big|\sum_{k \in \Z^n} \big(m^{\lambda,\rho}_1\big)\;\widecheck{}\;(x+k) \Big| \lesssim \big|\big(m^{\lambda,\rho}_1\big)\;\widecheck{}\;(x)\big| + \Big|\sum_{k \in \Z^n\setminus\{0\}} \big(m_1^{\lambda,\rho}\big)\;\widecheck{}\;(x+k) \Big|
\end{equation*}
Since $\check{\eta}$ vanishes to infinite order at the origin,~\eqref{spectral proof 12} implies that
\begin{equation*}
  \big|\big(m^{\lambda,\rho}_1\big)\;\widecheck{}\;(x)\big| = \big|\big(m^{\lambda,\rho}\big)\;\widecheck{}\;(x)\big(1- \check{\eta}(x)\big)\big| \lesssim \rho \lambda^{(n-1)/2}.
\end{equation*}
The remaining term satisfies the following, more restrictive, bound.

\begin{lemma}\label{crude kernel estimate} 
\begin{equation*}
    \Big|\sum_{k \in \Z^n\setminus\{0\}} \big(m_1^{\lambda,\rho}\big)\;\widecheck{}\;(x+k) \Big| \lesssim  (\lambda/\rho)^{(n-1)/2}.
\end{equation*}
\end{lemma}

\begin{proof} Since $|\big(m_1^{\lambda,\rho}\big)\;\widecheck{}\;(x)| \lesssim |\big(m^{\lambda,\rho}\big)\;\widecheck{}\;(x)|$, applying~\eqref{spectral proof 13} yields
\begin{equation*}
    \sum_{k \in \Z^n \setminus \{0\}}|\big(m_1^{\lambda,\rho}\big)\;\widecheck{}\;(x+k)| \lesssim \rho \lambda^{(n-1)/2}\sum_{k \in \Z^n \setminus \{0\}} |k|^{-(n-1)/2} (1 + \rho|k|)^{-N}\lesssim  (\lambda/\rho)^{(n-1)/2}.
\end{equation*}
\end{proof}
\noindent Combining these observations,~\eqref{spectral proof 7} immediately follows, concluding the proof of Proposition~\ref{spectral proposition}.
\end{proof}




\section{Improvements via multidimensional Weyl sum estimates}\label{Weyl sum section}

By Theorem~\ref{equivalence theorem} and the reductions in \S\ref{spectral projector section}, the uniform resolvent estimates in Theorem~\ref{refined resolvent theorem} are equivalent to the following multiplier bound.

\begin{proposition}\label{refined spectral proposition} Let $n \geq 3$, $\lambda \geq 1$ and $\varepsilon >0$. If $\rho := \lambda^{-\beta_n + \varepsilon}$, then
\begin{equation*}
    \big\| m^{\lambda, \rho}(D)f \big\|_{L^{\frac{2n}{n-2}}(\T^n)} \lesssim_{\varepsilon} \rho \lambda \|f\|_{L^{\frac{2n}{n+2}}(\T^n)}.
\end{equation*}
\end{proposition}

Proposition~\ref{refined spectral proposition} follows by combining the argument from \S\ref{proof section} with a more delicate estimation of the kernel. The use of the triangle inequality in the first step of the proof of Lemma~\ref{crude kernel estimate} introduces losses and the idea is to exploit cancellation between the terms of the sum. This is analogous to the refinements of Hlawka's argument found in~\cite{KN1992, Muller1999, Guo2012}. In particular, the exponential sum estimates from~\cite{Muller1999} imply the following strengthened version Lemma~\ref{crude kernel estimate}. 

\begin{lemma}\label{refined kernel lemma} Let $\lambda \geq 1$ and $1 \gtrsim \rho \geq \lambda^{-1}$. For all $q \in \N$ satisfying 
\begin{equation}\label{rho vrs lambda}
   \lambda \geq \rho^{-(q - 1 - 2/n + 2^{1-q})}  
\end{equation}
the kernel estimate
\begin{equation*}
    \Big|\sum_{k \in \Z^n\setminus\{0\}} \big(m_1^{\lambda,\rho}\big)\;\widecheck{}\;(x+k) \Big| \lesssim_{\varepsilon,q} \lambda^{\varepsilon}  (\rho^{q+1} \lambda)^{\omega_{n,q}}(\lambda/\rho)^{(n-1)/2}
\end{equation*}
holds for
\begin{equation*}
    \omega_{n,q} := \frac{n}{2n(2^q-1) +  2^{q+1}}.
\end{equation*}
\end{lemma}

Provided $\rho$ and $q$ are chosen so that $\rho^{q+1}$ is much smaller than $\lambda^{-1}$, this provides an improvement over the crude estimate from Lemma~\ref{crude kernel estimate}. 

Assuming Lemma~\ref{refined kernel lemma}, it is not difficult to adapt the argument of the previous section to prove the desired spectral projection bounds.

\begin{proof}[Proof (of Proposition~\ref{refined spectral proposition})] Let $q \in \N$ satisfy the hypotheses of Lemma~\ref{refined kernel lemma}. Arguing as before, Lemma~\ref{refined kernel lemma} implies that 
\begin{align*}
  \|m_1^{\lambda, \rho}(D)\|_{1 \to \infty} &\lesssim_{\varepsilon,q} \rho \lambda^{(n-1)/2} + \lambda^{\varepsilon}(\rho^{q+1} \lambda)^{\omega_{n,q}}(\lambda/\rho)^{(n-1)/2}\\
  &\lesssim \lambda^{\varepsilon}(\rho^{q+1} \lambda)^{\omega_{n,q}}(\lambda/\rho)^{(n-1)/2}.
\end{align*}
This refined estimate can be used in place of~\eqref{spectral proof 7} in the proof of Proposition~\ref{spectral proposition}. In particular, one deduces that
\begin{equation*}
    \|m^{\lambda,\rho}(D)\|_{p_1' \to p_1} \lesssim_{\varepsilon} \rho\lambda +  \lambda^{\varepsilon} (\rho^{q+1} \lambda)^{2\omega_{n,q}/n(n-1)} \rho^{1-3/n} \lambda^{1-1/n}
\end{equation*}
which provides an improved version of~\eqref{spectral proof 9}. In this case, one is led to the choice $\rho = \lambda^{-\beta_{n,q} + \varepsilon}$ where
\begin{equation*}
    \beta_{n,q} := \frac{1}{3} + \frac{n}{3} \cdot \frac{q-2}{3(n^2-1)2^q - q n - (3n-2)n} .
\end{equation*}
To optimise the estimate, $q$ should be chosen so as to make the exponent  as large as possible. Note that $\beta_{n,q} > 1/3$ whenever $q \geq 3$. Fixing $n$, a simple calculus exercise show that $\beta_{n,q}$ is a decreasing function for $q \geq 4$. Direct comparison between $\beta_{n,3}$ and $\beta_{n,4}$ then shows that $q = 3$ is always the optimal choice of parameter, if no additional constraint is imposed in the form of \eqref{rho vrs lambda}. However, it is not difficult to show that $\rho := \lambda^{-\beta_{n,3} + \varepsilon}$ automatically satisfies \eqref{rho vrs lambda}, provided $\varepsilon$ is sufficiently small. Since $\beta_n = \beta_{n,3}$, Proposition~\ref{refined spectral proposition} follows.
\end{proof}

It remains to prove Lemma~\ref{refined kernel lemma}. The argument uses two ingredients from~\cite{Muller1999}, the first of which is an elementary exponential sum bound.

\begin{theorem}[M\"uller~\cite{Muller1999}]\label{exponential sum theorem} Let $n,q \in \N$, $n \geq 2$, and $\lambda, M \geq 1$ satisfy 
\begin{equation}\label{lambda vrs M}
    \lambda \geq M^{q - 1 - 2/n + 2^{1-q}}.
\end{equation}
Suppose that $w \in C^{\infty}(\R^n)$ and $\phi \in C^{\infty}(\R^n)$ is real-valued and that these functions satisfy the following conditions:
\begin{enumerate}[i)]
    \item $\mathrm{supp}\,w$ is contained in $B(0, M)$;
    \item  $|\partial_u^{\alpha} w(u)| \lesssim_{\alpha} M^{-|\alpha|}$ and $|\partial_u^{\alpha}\phi(u)| \lesssim_{\alpha} \lambda M^{1-|\alpha|}$ for all $u \in \mathrm{supp}\,w$, $\alpha \in \N_0^n$;
    \item There exists some $\alpha(q) \in \N_0^n$ with $|\alpha(q)| = q$ such that 
    \begin{equation*}
        |\mathrm{Hess}\,\partial_u^{\alpha(q)} \phi(u)| \gtrsim (\lambda M^{-(q+1)})^{n} \qquad \textrm{for all $u \in \mathrm{supp}\,w$}.
    \end{equation*}
\end{enumerate}
Then there is a weighted exponential sum estimate
\begin{equation*}
    \Big|\sum_{k \in \Z^n} e^{2 \pi i \phi(k)} w(k) \Big| \lesssim_{\varepsilon} \lambda^{\varepsilon} M^{n} (M^{-(q+1)}\lambda)^{\omega_{n,q}}.
\end{equation*}
\end{theorem}

Here $\mathrm{Hess}$ is used to denote the Hessian determinant and, as before, $|\alpha| := \alpha_1 + \cdots + \alpha_n$.

For the phases and weights arising in the proof of Lemma~\ref{refined kernel lemma} it is straightforward to verify conditions i) and ii) of Theorem~\ref{exponential sum theorem}.  Condition iii), however, only holds locally and after applying a linear coordinate transformation. The existence of such a coordinate transformation is the second ingredient from~\cite{Muller1999}. 

\begin{lemma}[M\"uller~\cite{Muller1999}]\label{preparation lemma} For $n, q \in \N$, $n \geq 2$ there exist open regions $S_{\ell} \subset \R^n \setminus \{0\}$ and integer matrices $Q_{\ell} \in \mathrm{GL}(n,\R)$ for $1 \leq i \leq L = L(n,q) \in \N$ with the following properties:
\begin{enumerate}[i)]
    \item $\R^n \setminus \{0\} \subseteq \bigcup_{i=1}^L S_{\ell}$ and if $x \in S_{\ell}$ and $\lambda > 0$, then $\lambda x \in S_{\ell}$;
    \item The function $\Phi_{\ell} \colon \R^n \to \R$ given by $\Phi_{\ell}(u):= | Q_{\ell} u|$ satisfies 
\begin{equation*}
    \Big|\mathrm{Hess}\, \frac{\partial^{q}\Phi_{\ell}}{\partial u_1 \partial u_n^{q-1}}(u) \Big| \gtrsim |u|^{-(q+1)n} \qquad \textrm{for all $u \in Q_{\ell}^{-1}S_{\ell}$}.
\end{equation*}
\end{enumerate}

\end{lemma}

This follows from~\cite[Lemma 3]{Muller1999}. In particular, it suffices to find an open covering of the unit sphere (rather than the whole of $\R^n \setminus \{0\}$) satisfying property ii), since the full result then follows by homogeneity. The desired cover can then be obtained by combining~\cite[Lemma 3]{Muller1999} with a compactness argument.

\begin{proof}[Proof (of Lemma~\ref{refined kernel lemma})] The proof is similar to that of Theorem 1 in~\cite{Muller1999}. 

By \eqref{spectral proof 12}, one may write
\begin{equation*}
    \big(m_1^{\lambda,\rho}\big)\;\widecheck{}\;(x) =  \rho \sum_{\pm} e^{\pm 2 \pi i \lambda |x|}I_{\pm}^{\lambda,\rho}(x)
\end{equation*}
where
\begin{equation*}
I_{\pm}^{\lambda, \rho}(x) := \int_0^{\infty} e^{\pm 2 \pi i(r - \rho^{-1}\lambda)\rho|x|} a_{\pm}(\rho rx) \beta\big(r - \rho^{-1}\lambda\big) (\rho r)^{n-1}\,\ud r \cdot (1 - \check{\eta}(x)).
\end{equation*}
Applying integration-by-parts as in \eqref{spectral proof 13}, it follows that 
\begin{equation}\label{refined proof 1}
   |\partial_x^{\alpha} I_{\pm}^{\lambda,\rho}(x)| \lesssim_{\alpha}  \frac{ \lambda^{(n-1)/2}}{|x|^{(n-1)/2 + |\alpha|}} (1 + \rho|x|)^{-N} \quad \textrm{for all $\alpha \in \N_0^n$}
\end{equation}
where $N := \lceil 100n \varepsilon^{-1} \rceil $. Note that this is a substantially larger (but still admissible) choice of $N$ than that used in the previous arguments. With this choice, it follows, for instance, that $ (1 + \rho|x|)^{-N} \lesssim_{\varepsilon} \rho^{100n}$ whenever $|x| > \rho^{-1-\varepsilon}$.

Since the functions $I_{\pm}^{\lambda, \rho}$ decay rapidly when $|x| \geq \rho^{-1}$, it suffices to show that
\begin{equation}\label{refined proof 2}
    \sup_{x \in [1/2,1/2]^n} \Big|\sum_{ \substack{k \in \Z^n \setminus \{0\} \\ |x+k| \leq \rho^{-1-\varepsilon}}} e^{2 \pi i\lambda|x+k|} I_{\pm}^{\lambda,\rho}(x + k) \Big| \lesssim_{\varepsilon} \lambda^{\varepsilon} \rho^{-1} (\rho^{q+1} \lambda)^{\omega_{n,q}} (\lambda/\rho)^{(n-1)/2}
\end{equation}
holds for all $q \in \N$ satisfying \eqref{rho vrs lambda}. The support of the weight functions $I^{\lambda,\rho}_{\pm}$ are decomposed dyadically by writing
\begin{equation*}
    I^{\lambda,\rho}_{\pm} = \sum_{j \in \Z} I^{\lambda,\rho}_{\pm, j} \qquad \textrm{where} \qquad I^{\lambda,\rho}_{\pm, j}(x) := I^{\lambda,\rho}_{\pm}(x)\zeta(2^{-j}|x|)
\end{equation*}
for a suitable choice of $\zeta \in C^{\infty}_c(\R)$ satisfying $\mathrm{supp}\, \zeta \subseteq [1/2,2]$. For any fixed value of $x \in [-1/2,1/2]^n$ there are only $O(\log \rho^{-1})$ values of $j$ for which $I^{\lambda,\rho}_{\pm, j}(x+k)$ is non-zero as $k$ varies over all $k \in \Z^n \setminus \{0\}$ satisfying $|x+k| \leq \rho^{-1-\varepsilon}$. Thus, by dyadic pigeonholing, it suffices to show  \eqref{refined proof 2} holds with $I^{\lambda,\rho}_{\pm}$ replaced with $I^{\lambda,\rho}_{\pm, j}$ for some fixed choice of $j$ satisfying $1 \lesssim 2^j \lesssim \rho^{-1-\varepsilon}$.

Fix $q \in \N$ satisfying \eqref{rho vrs lambda} and a choice of sign $\pm$ and let 
\begin{equation*}
w^{\lambda,j}(u) := \lambda^{-(n-1)/2}2^{j(n-1)/2}I_{\pm,j}^{\lambda,\rho}(u)  \qquad \textrm{and} \qquad \phi^{\lambda}(u) :=  \pm \lambda |u|.
\end{equation*}
Given any $x \in \R^n$, define the translates
\begin{equation*}
    w_x^{\lambda,j}(u) := w^{\lambda,j}(x + u) \qquad \textrm{and} \qquad  \phi_x^{\lambda}(u) := \phi^{\lambda}(x + u)
\end{equation*} 
 and observe that, by \eqref{refined proof 1}, if $u \in \mathrm{supp}\,w_x^{\lambda,j}$, then
\begin{equation}\label{refined proof 3}
    |\partial_u^{\alpha} w_x^{\lambda,j}(u)| \lesssim_{\alpha} 2^{-j|\alpha|} \quad \textrm{and} \quad |\partial_u^{\alpha}\phi_x^{\lambda}(u)| \lesssim_{\alpha} \lambda 2^{j(1-|\alpha|)} \quad \textrm{for all $\alpha \in \N_0^n$.}
\end{equation}
Thus, in view of the above reductions, it suffices to show that
\begin{equation}\label{refined proof 4}
    \sup_{x \in \R^n} \Big|\sum_{ k \in \Z^n} e^{2 \pi i\phi^{\lambda}_x(k)} w_x^{\lambda,j}(k) \Big| \lesssim_{\varepsilon, q} \lambda^{\varepsilon} 2^{j n} (2^{-j(q+1)} \lambda)^{\omega_{n,q}}.
\end{equation}
Note that the reduction in \eqref{refined proof 4} relies upon the (readily checked) fact 
\begin{equation*}
    \frac{n+1}{2} - (q+1)\omega_{n,q} > 0 \quad \textrm{for all $n, q \in \N$ with $n \geq 2$}
\end{equation*}
which, in particular, implies that
\begin{equation*}
    2^{j(n+1)/2} 2^{-j(q+1)\omega_{n,q}} \lesssim \rho^{-O(\varepsilon)}  \rho^{-1}\rho^{(q+1)\omega_{n,q}}\rho^{-(n-1)/2}.
\end{equation*}
The estimate \eqref{refined proof 4} will follow from Theorem~\ref{exponential sum theorem}, although some preparatory steps are needed to ensure the conditions of the theorem hold in this case.

Let $S_{\ell} \subset \R^n\setminus\{0\}$ and $Q_{\ell} \in \mathrm{GL}(n,\R)$ for $1 \leq \ell \leq L$ be open sets and integer matrices, respectively, satisfying the properties i) and ii) from Lemma~\ref{preparation lemma}. By forming a homogeneous partition of unity adapted to the $(S_{\ell})_{\ell=1}^L$ and pigeonholing, it suffices to show that
\begin{equation*}
    \sup_{x \in \R^n} \Big|\sum_{k \in \Z^n} e^{2 \pi i \phi^{\lambda}_x(k)}w^{\lambda,j}_x(k)\psi_x(k)\Big| \lesssim \lambda^{\varepsilon}2^{jn} (2^{-j(q+1)}\lambda)^{\omega_{n,q}},
\end{equation*}
where $\psi_x(u) := \psi(x + u)$ for $\psi \in C^{\infty}(\R^n \setminus\{0\})$ real-valued, homogeneous of degree 0 and supported in $S := S_{\ell_0}$ for some $1 \leq \ell_0 \leq L$.  

Let $Q := Q_{\ell_0}$ and note that the lattice $Q \Z^n$ is a finite index subgroup of $\Z^n$. Thus, there exist some $\mathcal{B} \subseteq \Z^n$ with $\# \mathcal{B} \lesssim_q 1$ such that
\begin{equation*}
    \Z^n = \bigcup_{b \in \mathcal{B}} (b + Q \Z^n),
\end{equation*}
where the union is disjoint. Fix $b \in \mathcal{B}$ and write
\begin{equation*}
    \tilde{\phi}_x^{\lambda}(u) := \phi_x^{\lambda}(b + Q u) \qquad \textrm{and} \qquad \tilde{w}_x^{\lambda,j}(u) :=w_x^{\lambda,j}(b+ Q u)\psi_x(b + Q u).
\end{equation*}
Once again by pigeonholing, the desired estimate would follow from
\begin{equation*}
    \sup_{x \in \R^n} \Big|\sum_{k \in \Z^n} e^{2 \pi i \tilde{\phi}^{\lambda}_x(k)}\tilde{w}^{\lambda,j}_x(k)\Big| \lesssim \lambda^{\varepsilon}2^{j n} (2^{-j(q+1)}\lambda)^{\omega_{n,q}}.
\end{equation*}

To conclude the proof, it suffices to show that, for any $x \in \R^n$, the functions $\tilde{\phi}^{\lambda}_x$ and $\tilde{w}^{\lambda,j}_x$ satisfy the hypotheses of Theorem~\ref{exponential sum theorem} with $M \sim_q 2^j$ and $\alpha(q) := (1, 0, \dots, q-1)$; since $q$ is chosen so as to satisfy \eqref{rho vrs lambda}, one may safely assume \eqref{lambda vrs M} holds for such a choice of $M$. Clearly the support condition i) holds. By \eqref{refined proof 3} and the homogeneity of $\psi$, it follows that
\begin{equation*}
    |\partial_u^{\alpha} \tilde{w}^{\lambda,j}(u)| \lesssim_{\alpha} 2^{-j|\alpha|} \quad \textrm{and} \quad |\partial_u^{\alpha}\tilde{\phi}^{\lambda}(u)| \lesssim_{\alpha} \lambda 2^{j(1-|\alpha|)} \quad \textrm{for all $\alpha \in \N_0^n$,}
\end{equation*}
which is condition ii). Finally, Lemma~\ref{preparation lemma} ensures that
\begin{equation*}
  \big|\mathrm{Hess}\,\partial^{\alpha(q)}_u\tilde{\phi}^{\lambda}_x(u) \big| \gtrsim (\lambda 2^{-j(q+1)})^n \qquad \textrm{for all $u \in \mathrm{supp}\,\tilde{w}_x^{\lambda,j}$}.
\end{equation*}
Indeed, if $u \in \mathrm{supp}\,\tilde{w}_x^{\lambda,j}$, then $x + b + Q u \in S$ and so $\tilde{x} + u \in Q^{-1}S$ for $\tilde{x} := Q^{-1}(x+b)$. If $\Phi(u) := |Q u|$, then $\tilde{\phi}^{\lambda}_x(u) = \pm \lambda \Phi(\tilde{x} + u)$ and so Lemma~\ref{preparation lemma} implies that
\begin{equation*}
    |\mathrm{Hess}\,\partial^{\alpha(q)}_u \tilde{\phi}^{\lambda}_x(u)| = \lambda^n|\mathrm{Hess}\,\partial^{\alpha(q)}_u \Phi(\tilde{x} + u)| \gtrsim \lambda^n 2^{-(q+1)n},
\end{equation*}
as required. 
\end{proof}




\bibliography{Reference}
\bibliographystyle{amsplain}

\end{document}